\documentclass[11pt,reqno]{amsart}
\usepackage{amssymb,latexsym,enumerate}
\usepackage{graphicx,subfig}
\usepackage{epsfig}
\usepackage{color}

\usepackage{lineno}
\newcommand*\patchAmsMathEnvironmentForLineno[1]{%
  \expandafter\let\csname old#1\expandafter\endcsname\csname #1\endcsname
  \expandafter\let\csname oldend#1\expandafter\endcsname\csname end#1\endcsname
  \renewenvironment{#1}%
     {\linenomath\csname old#1\endcsname}%
     {\csname oldend#1\endcsname\endlinenomath}}%
\newcommand*\patchBothAmsMathEnvironmentsForLineno[1]{%
  \patchAmsMathEnvironmentForLineno{#1}%
  \patchAmsMathEnvironmentForLineno{#1*}}%
\AtBeginDocument{%
\patchBothAmsMathEnvironmentsForLineno{equation}%
\patchBothAmsMathEnvironmentsForLineno{align}%
\patchBothAmsMathEnvironmentsForLineno{flalign}%
\patchBothAmsMathEnvironmentsForLineno{alignat}%
\patchBothAmsMathEnvironmentsForLineno{gather}%
\patchBothAmsMathEnvironmentsForLineno{multline}%
}

\newtheorem{thm}{Theorem}[section]
\newtheorem{dfn}[thm]{Definition}
\newtheorem{lma}[thm]{Lemma}
\newtheorem{cor}[thm]{Corollary}
\newtheorem{prp}[thm]{Proposition}
\newtheorem{exm}[thm]{Example}
\newtheorem{clm}[thm]{Claim}

\newtheorem{conj}[thm]{Conjecture}
\newtheorem{fact}[thm]{Fact}

\DeclareMathOperator{\ex}{ex}

\title{${\ell}$-degree Tur\'{a}n density}

\author{Allan Lo}
\address{School of Mathematics\\ University of Birmingham\\Birmingham\\ B15 2TT\\ UK}
\thanks {The research leading to these results was partially supported by the  European Research Council
under the ERC Grant Agreement no. 258345 (A. Lo).}
\email{s.a.lo@bham.ac.uk}

\author{Klas Markstr{\"o}m}
\address{Department of Mathematics and Mathematical Statistics\\ Ume\r{a} University\\ S-901 87 Ume\r{a}\\ Sweden}
\email{klas.markstrom@math.umu.se}

\date{\today}
\keywords{extremal problem, hypergraph, Tur\'{a}n density, jump constant, $\ell$-degree}
\subjclass[2010]{05C35}

\begin{document}

\numberwithin{equation}{section}


\begin{abstract}
Let $H_n$ be a $k$-graph on $n$ vertices. 
For $0 \le \ell <k$ and an $\ell$-subset $T$ of $V(H_n)$, define the degree $\deg(T)$ of $T$ to be the number of $(k-\ell)$-subsets~$S$ such that $S \cup T$ is an edge in~$H_n$.
Let the minimum $\ell$-degree of $H_n$ be $\delta_{\ell}(H_n) = \min \{ \deg(T) : T \subseteq V(H_n)$ and $|T|=\ell\}$.
Given a family $\mathcal{F}$ of $k$-graphs, the $\ell$-degree Tur\'an number $\ex_{\ell}(n, \mathcal{F})$ is the largest $\delta_{\ell}(H_n)$ over all $\mathcal{F}$-free $k$-graphs $H_n$ on $n$ vertices. 
Hence, $\ex_0(n, \mathcal{F})$ is the Tur\'an number.
We define $\ell$-degree Tur\'an density to be $$\pi^k_{\ell}(\mathcal{F}) = \limsup_{n \rightarrow \infty} \frac{\ex_{\ell}(n, \mathcal{F} )}{ \binom{n- \ell}{k}}.$$
In this paper, we show that for $k> \ell >1$, the set of $\pi_{\ell}^k(\mathcal{F})$ is dense in the interval $[0,1)$.
Hence, there is no ``jump'' for $\ell$-degree Tur\'an density when $k>\ell >1$.
We also give a lower bound on $\pi_{\ell}^k(\mathcal{F})$ in terms of an ordinary Tur\'an density.
\end{abstract}

\maketitle

\section{Introduction}

A \emph{$k$-uniform hypergraph}, or a \emph{$k$-graph} for short, is a pair $H = (V(H),E(H))$, where $V(H)$ is a finite set of vertices and the edge set $E(H)$ is a set of $k$-subsets of $V(H)$. 
The notation $H_n$ indicates that $|V(H_n)| = n$.

Given a family $\mathcal{F}$ of $k$-graphs, a $k$-graph $H_n$ is \emph{$\mathcal{F}$-free} if it contains no copy of any member of $\mathcal{F}$.
The \emph{Tur\'{a}n number $\ex(n,\mathcal{F})$} is the maximum number of edges in $H_n$ for all $\mathcal{F}$-free $k$-graphs $H_n$ (with $n$ vertices).
Katona, Nemetz and Simonovits~\cite{MR0172263} showed that the \emph{Tur\'{a}n density} $\pi(\mathcal{F}) = \lim_{n \rightarrow \infty} \ex(n,\mathcal{F})/ \binom{n}{k}$ exists.
The celebrated Erd\H{o}s--Simonovits--Stone theorem~\cite{MR0205876,MR0018807} implies that for $k=2$ and any family $\mathcal{F}$ of $2$-graphs, $\pi ( \mathcal{F} ) = \min \left\{ 1 - \frac{1}{\chi (F) -1} : F \in \mathcal{F} \right\}$, where $\chi(G)$ is the chromatic number of~$G$.
For $k \ge 3$, determining $\pi(\mathcal{F})$ is known to be hard.

Recall that the Tur\'{a}n number $\ex(n,\mathcal{F})$ studies the maximum number of edges in an $\mathcal{F}$-free $k$-graph of order~$n$. 
One natural variant is the so-called Dirac-type condition, which studies the largest possible minimum degree for $\mathcal{F}$-free $k$-graphs. 
For a $k$-graph $H$ and an $\ell$-subset~$T$ of~$V(H)$, let the \emph{neighbourhood of $T$} be $N^H(T) = \{S \in \binom{ V (H) }{k-\ell} : T \cup S \in E(H)\}$ and let the \emph{degree of $T$} be $\deg^H(T) = |N^H(T)|$.
(Here, $\binom{V}{k}$ is the set of $k$-subsets of $V$.)
The \emph{minimum $\ell$-degree $\delta_{\ell}(H)$} is $\min\{ \deg^H(T) : T \in \binom{V(H)}{ \ell }\} $.
Hence, if $\ell \ge k$, then $\delta_{\ell}(H) = 0$ unless $H$ is complete and $k= \ell$.
Recently, there has been an increasing interest in Dirac-type conditions for perfect matchings and Hamiltonian cycles. 
We recommend \cite{rodldirac} for a survey on perfect matchings and Hamiltonian cycles in $k$-graphs.

Note that $\delta_0(H) = |E(H)|$.
Thus, $\ex(n,\mathcal{F})$ is the maximum $\delta_0(H_n)$ over all $\mathcal{F}$-free $k$-graphs $H_n$.
In this paper, we generalize $\ex(n,\mathcal{F})$ for the minimum $\ell$-degree.
Define \emph{$\ell$-degree Tur\'an number $\ex_{\ell}(n,\mathcal{F})$} to be the maximum $\delta_{\ell}(H_n)$ over all $\mathcal{F}$-free $k$-graphs $H_n$.
We further define the \emph{$\ell$-degree Tur\'an density of $\mathcal{F}$} to be
\begin{align}
\pi_{\ell}(\mathcal{F}) = \limsup_{n \rightarrow \infty} \frac{\ex_{\ell}(n,\mathcal{F})}{\binom{n}{k-\ell}}. \nonumber
\end{align}
As mentioned by Mubayi and Zhao~\cite{MR2337241} for the case $\ell = k-1$, one should divide by $\binom{n-k}{k-\ell}$ instead of $\binom{n}{k-\ell}$ in the definition above.
Since $\ell$ and $k$ are fixed and $n$ tends to infinity, this does not affect the value of $\pi_{\ell}(\mathcal{F})$. 

Note that for $ \ell = 0$, we have $\ex(n,\mathcal{F}) = \ex_0(n,\mathcal{F})$ and $\pi_0(\mathcal{F}) = \pi(\mathcal{F})$.
For $ \ell =1$, it is not difficult to deduce that the limit $\pi_1(\mathcal{F})  = \lim_{n \rightarrow \infty} \ex_1(n,\mathcal{F})/\binom{n}{k-1} = \pi(\mathcal{F})$.
The exact value of $\pi ( \mathcal{F} )$ is determined only for a few families $\mathcal{F}$ of $k$-graphs.
We recommend \cite{MR1341481, keevash2011hypergraph} for surveys on hypergraph Tur\'an problems.

By a simple averaging argument, we know that for $0 \le \ell \le \ell' < k$ and $k$-graphs $H_n$, if $\delta_{\ell'}(H_n) \ge \alpha \binom{n}{k-\ell'}$ then $\delta_{\ell}(H_n) \ge (\alpha-o(1)) \binom{n}{k-\ell}$.
Thus, $\pi_{ \ell }(\mathcal{F}) \ge \pi_{\ell'}(\mathcal{F}) $ for $0 \le \ell \le \ell'<k$.
Recall that for a $k$-graph $H$ and an $\ell$-subset~$T$ of~$V(H)$, $N^H(T) = \{S \in \binom{V(H)}{k-\ell} : T \cup S \in E(H)\}$.
Thus, $N^H(T)$ can be viewed as a $(k-\ell)$-graph.
Often $N^H(T)$ is called the \emph{link-graph of}~$T$.
Define the \emph{$\ell$-link family $\mathcal{ L }_{\ell}(H)$ of $H$} to be the family of the $N^H(T)$ for all $\ell$-subsets~$T$ of~$V(H)$.
For example, $\mathcal{ L }_{\ell}(K_t^k) = \left\{ K_{ t - \ell }^{k-\ell} \right\}$, where $K_t^k$ is the complete $k$-graphs on $t$ vertices.
Also, $\mathcal{ L }_1(K_4^3-e) = \{ K_{3} , P_3 \}$, where $K_4^3-e$ is the unique 3-graph on 4 vertices with 3 edges and $P_3$ is the path of length~2.
For a family $\mathcal{F}$ of $k$-graphs, $\mathcal{ L }_{\ell}(\mathcal{F})$ is simply the union of $\mathcal{ L }_{\ell}(F)$ for all $F \in \mathcal{F}$.
Using a probabilistic construction, we show that $\pi_{\ell}( \mathcal{F} ) \ge \pi( \mathcal{ L }_{ \ell - 1 }(\mathcal{F} ) )$.
\begin{prp} \label{prp:lowerbound}
Let $2 \le \ell < k $ be integers and let $\mathcal{F}$ be a family of $k$-graphs.
Then $\pi_{ \ell }(\mathcal{F}) \ge \pi(\mathcal{ L }_{ \ell - 1 }(\mathcal{F}))$.
\end{prp}

\begin{cor}
For $2 \le \ell < k \le t$, $\pi_{ \ell }(K_t^k) \ge \pi_{1}(K_{t- \ell +1}^{k-\ell+1})$.
In particular, $\pi_{k-1}(K_t^k) \ge (t-k)/(t-k+1)$ and $\pi_{k-1}(K_{k+1}^k) \ge 1/2$.
\end{cor}

Note that for $k=3$, the corollary above gives $\pi_2(K_4^3) \ge 1/2$.
In fact, Czygrinow and Nagle~\cite{MR1829685} conjectured that equality holds.
For $k = 4$, the corollary also implies that $\pi_3(K_5^4) \ge 1/2$.
Here, we also give a different construction to show that $\pi_3(K_5^4) \ge 1/2$, which is due to Giraud~\cite{MR1077144}. 

\begin{exm}
Let $M$ be an $m \times m$ $0/1$-matrix. 
Define the hypergraph $H(M)$ with $2m$ vertices corresponding to the rows and columns of $M$.
Any $4$-set consisting of exactly $3$ rows or $3$ columns is an edge.
Also any $4$-set of $2$ rows and $2$ columns inducing a $2 \times 2$ submatrix with odd sum is an edge.
\end{exm}

It is easy to verify that $H(M)$ is $K_5^4$-free and $\delta_3(H(M)) \ge (1/2 - o(1)) |V(H(M))|$.
When $M$ is the Hadamard matrix and replacing the $-1$'s by $0$'s, $H(M)$ was originally used to show that $\pi (K_5^4) \ge 11/16$.
Sidorenko~\cite{MR1341481} conjectured that $\pi (K_5^4) = 11/16$.
The best known upper bound is $\pi (K_5^4) \le \frac{1753}{2380} \approx 0.73655 \dots$ given by the second author~\cite{MR2548922}.
We believe that it would be very interesting to get an answer to the following conjecture.

\begin{conj}
For all $k \ge 2$, $\pi_{k-1}(K_{k+1}^k) = 1/2$.
\end{conj}

Note that $(\ex ( n , \mathcal{F}) / \binom{n}{k})_{n \in \mathbb{N}}$ is a decreasing sequence for all families $\mathcal{F}$ of $k$-graphs.
Mubayi and Zhao~\cite{MR2337241} commented that they could not prove that $(\ex_{k-1}(n,\mathcal{F})/{n})_{n \in \mathbb{N}}$ (or $(\ex_{k-1}(n,\mathcal{F})/ (n-k+1) )_{n \in \mathbb{N}}$) is a decreasing sequence for $k \ge 3$.
Here, we give an explicit example to show that the sequence is not necessarily monotone.
With the aid of computer, we determined the exact values of $\ex_2 (n,K_4^3)$ for $n \le 11$. 
The exact values are, with the pairs being $(n,\ex_2 (n,K_4^3))$, $\{ (5,2) , (6,3), (7,4), (8,4), (9,5), (10,5), (11,6) \} $.
Thus, neither $(\ex_{2}(n,K_4^3)/{n})_{n \in \mathbb{N}}$ nor $(\ex_{2}(n,K_4^3)/(n-2))_{n \in \mathbb{N}}$ is a decreasing sequence.
Although $(\ex_{k-1}(n,\mathcal{F})/{n})_{n \in \mathbb{N}}$ is not monotone, Mubayi and Zhao~\cite{MR2337241} showed that this sequence does indeed converge to $\pi_{k-1} ( \mathcal{F} )$.
We generalize their result and show that $\pi_{\ell}(\mathcal{F}) = \lim_{n \rightarrow \infty} \ex_{\ell}(n, \mathcal{F}) / \binom{n}{k-\ell}$ for all $k > \ell \ge 0$ and all families $\mathcal{F}$ of $k$-graphs.
This observation has been made before (\cite{keevash2011hypergraph} page 118) but that we have not found a written proof of it.

\begin{prp} \label{prp:existence}
Given integers $k > \ell \ge 0$, $\pi_{\ell}(\mathcal{F}) = \lim_{n \rightarrow \infty} \ex_{\ell}(n, \mathcal{F}) / \binom{n}{k-\ell}$ for all families $\mathcal{F}$ of $k$-graphs.
\end{prp}

It is not surprising that for a $k$-graph $F$, $\pi_{\ell}( F)$ also satisfies the so-called ``supersaturation'' phenomenon discovered by Erd\H{o}s and Simonovits~\cite{MR726456}.
Informally speaking, given a $k$-graph $F$ of order~$f$ and $\varepsilon >0$, there exist $\delta>0$ and $N$ such that every $k$-graph $H_n$ with $n \ge N$ and $\delta_{\ell}(H_n) \ge (\pi_{\ell}( F) + \varepsilon) \binom{n}{k-\ell}$ contains $\delta \binom{n}{f}$ copies of $F$.
By the supersaturation phenomenon, ``blowing-up $F$'' does not change the value of $\pi_{\ell}(F)$.
Given an integer~$s$ and a $k$-graph~$F$, the \emph{$s$-blow-up $F(s)$ of $F$} is the $k$-graph $(V', E')$ such that $V'$ is obtained by replacing each $v_i \in V(F)$ by a vertex set $V_i$ of size $s$, and $E' = \{ u_1 \dots u_k : u_j \in V_{i_j}, v_{i_1} \dots v_{i_k} \in E(F) \}$.
Given a family of $k$-graphs, let $\mathcal{F}(s) = \{ F(s) : F \in \mathcal{F} \}$.

\begin{prp}[Supersaturation] \label{prp:supersaturation}
Let $k > \ell \ge 1$ be integers.
Let $\mathcal{F}$ be a family of $k$-graphs.
For any $\varepsilon >0$, there exist $\delta = \delta ( \varepsilon, \mathcal{F} ) >0$ and $N$ such that every $k$-graph $H_n$ with $n >N$ and $\delta_{\ell}(H_n) > (\pi_{\ell}(\mathcal{F}) + \varepsilon ) \binom{n}{k-\ell}$ contains $\delta \binom{n}{f}$ copies of $F$ for some $F \in \mathcal{F}$ on~$f$ vertices.
Moreover, $\pi_{\ell}(\mathcal{F}) = \pi_{\ell}(\mathcal{F}(s))$ for all integers $s \ge 1$.
\end{prp}

\subsection{Jumps}

For $0\le a < b \le 1$, we denote by $[a,b)$ and $(a,b)$ the intervals $\{ c: a \le c < b\}$ and $\{ c: a < c < b\}$, respectively.
For $k > \ell \ge 0$, let
\begin{align*}
\Pi_{\ell}^k = \{ \pi_{\ell}( \mathcal{F} ): \textrm{ $\mathcal{F}$ is a family of $k$-graphs}\}.
\end{align*}
The Erd\H{o}s-Simonovits-Stone theorem implies that $\Pi_0^2 = \{0, 1/2, 2/3, \dots\}$.
The constant $\alpha \in [0,1)$ is \emph{a jump for $k$}, if there exists a constant $\delta > 0$ such that no family $\mathcal{F}$ of $k$-graphs satisfies $\pi ( \mathcal{F} ) \in (\alpha, \alpha + \delta )$.
Hence, every $\alpha \in [0,1)$ is a jump for~$k=2$.
For $k \ge 3$, every $\alpha \in [0,k!/k^k)$ is a jump for $k$ by a result of Erd\H{o}s~\cite{MR0183654}.
However, Frankl and R\"odl~\cite{MR771722} showed that $1 - t^{1-k}$ is not a jump for $t > 2k$ and $k \ge 3$.
Later, other non-jump values were discovered, e.g. \cite{MR2290321, MR2600475}.
On the other hand, the first example of jumps in $[k!/k^k,1)$ for $k \ge 3$ was only given recently by Baber and Talbot~\cite{MR2769186}.

We generalize the concept of jump as follows.
Given $k \ge \ell \ge 0$, the constant $\alpha \in [0,1)$ is \emph{a $\pi_{\ell}^k$-jump}, if there exists a constant $\delta > 0$ such that no family $\mathcal{F}$ of $k$-graphs satisfies $\pi_{\ell} ( \mathcal{F} ) \in (\alpha, \alpha + \delta )$.
Equivalently, $\alpha \in [0,1)$ is a $\pi_{\ell}^k$-jump if there exists a constant $\delta > 0$ such that $(\alpha, \alpha + \delta) \cap \Pi_{\ell}^k = \emptyset$.
Recall that $\pi_0(\mathcal{F}) = \pi_1(\mathcal{F})$ for every family $\mathcal{F}$ of $k$-graphs, so for $k =2$ any $\alpha \in [0,1)$ is a $\pi_{\ell}^2$-jump.
Mubayi and Zhao~\cite{MR2337241} showed that no $\alpha \in [0,1)$ is a $\pi_{k-1}^k$-jump for all $k \ge 3$.
We generalize their construction to show that no $\alpha \in [0,1)$ is a $\pi_{ \ell }^k$-jump for all $k> \ell >1$.

\begin{thm} \label{thm:nojump}
For all integers $k > \ell > 1$, no $\alpha \in [0,1)$ is a $\pi_{\ell}^k$-jump.
In particular, $\Pi_{\ell}^k$ is dense in $[0,1)$.
\end{thm}

This answers a question of Mubayi, Pikhurko and Sudakov~\cite{Pik}.
Furthermore, we would like to ask the question whether $\Pi_{\ell}^k = [0,1)$.

We prove Propositions~\ref{prp:existence} and~\ref{prp:supersaturation} in the next section.
Proposition~\ref{prp:lowerbound} is proved in Section~\ref{sec:lowerbound}.
Finally, we prove Theorem~\ref{thm:nojump} in Sections~\ref{sec:0} and~\ref{sec:nojump}.

\section{Supersaturation} \label{sec:supersaturation}

Our aim for this section is to prove Propositions~\ref{prp:existence} and~\ref{prp:supersaturation}.
For $n \in \mathbb{N}$, we denote by $[n]$ the set $\{1, 2, \dots, n \}$.
Given a $k$-graph $H$ and $U\subseteq V(H)$, write $G[U]$ to be the induced $k$-subgraph of $H$ on~$U$.
We will need Azuma's inequality (see e.g.~\cite{MR1885388}).

\begin{thm}[Azuma's inequality] \label{thm:azuma}
Let $\{ X_i: i=0, 1, \dots\}$ be a martingale and $|X_i - X_{i-1}|  \le c_i$.
Then for all positive integers $N$ and $\lambda >0$,
\begin{align*}
\mathbb{P}(X_N \le X_0 - \lambda) \le \exp \left( \frac{-\lambda^2}{2 \sum_{i=1}^N c_i^2}\right).
\end{align*}
\end{thm}

The following lemma is important.

\begin{lma} \label{lma:subgraph}
Let $k > \ell \ge 1$.
Given $\alpha, \varepsilon >0$ with $\alpha + \varepsilon <1$, let $M(k,\ell,\varepsilon)$ be the smallest integer such that 
\begin{align*}
\binom{m}{ \ell } \exp \left( - \frac{\varepsilon^2 m }{8(k-\ell)^2} \right) \le 1/2
\textrm{ and }
\frac{ \binom{m}{k-\ell} - \binom{ m - \ell }{k-\ell} }{ \binom{m}{k-\ell} - \binom{ m - \ell }{k-\ell}/2 } \le \varepsilon
\end{align*}
for all $m > M(k,\ell,\varepsilon)$.
If $n \ge m \ge M(k,\ell,\varepsilon)$ and $H$ is a $k$-graph on $[n]$ with $\delta_{\ell} (H) \ge (\alpha + \varepsilon) \binom{n}{k-\ell}$, then the number of $m$-subsets~$S$ of~$[n]$ satisfying $\delta_{\ell}(H[S]) > \alpha \binom{m}{k-\ell}$ is at least $\frac12 \binom{n}{m}$.
\end{lma}

\begin{proof}

Given an $\ell$-subset~$T$ of~$[n]$, we call an $m$-subset~$S$ of~$[n]$ \textit{bad for} $T$ if $T \subseteq S$ and $\left| N(T) \cap \binom{S}{k-\ell} \right| \le \alpha \binom{m}{k-\ell}$.
An $m$-subset~$S$ of~$[n]$ is called \textit{bad} if it is bad for some~$T$.
In the next claim, we count the number of $m$-subsets~$S$ that is bad for a given~$T$.

\begin{clm}
Let $\phi_T$ be the number of $m$-subsets $S$ that are bad for an $\ell$-subset~$T$ of~$[n]$.
Then
\begin{align*}
	\phi_T
	\le \binom{ n - \ell }{ m - \ell } \exp \left( - \frac{\varepsilon^2 m }{8(k-\ell)^2} \right).
\end{align*}
\end{clm}

\begin{proof}[Proof of claim]
Observe that
\begin{align*}
	\phi_T = \left| \left\{  S' \in \binom{[n] \setminus T}{ m - \ell } : \left| N(T) \cap \binom{S'}{k-\ell} \right| \le \alpha \binom{m}{k-\ell}\right\} \right|.
\end{align*}
Let $X$ be the random variable $\left| N(T) \cap \binom{S'}{k-\ell} \right|$, where $S'$ is an $(m-\ell)$-subset of~$[n] \setminus T$ picked uniformly at random.
We consider the vertex exposure martingale on $S'$.
Let $Z_i$ be the $i$th exposed vertex in $S'$.
Define $X_i = \mathbb{E}(X|Z_1, \dots, Z_i)$.
Note that $\{ X_i: i=0, 1, \dots, m - \ell \}$ is a martingale and $X_0 \ge \left( \alpha + \varepsilon \right) \binom{ m - \ell }{k-\ell} \ge \left( \alpha + \varepsilon/2 \right) \binom{m}{k-\ell}$.
Moreover, $|X_i - X_{i-1}| \le  \binom{i-1}{k-\ell-1}$.
Thus, by Theorem~\ref{thm:azuma} taking $\lambda = \varepsilon \binom{m}{k-\ell}/2$ and $c_i = \binom{i-1}{k-\ell-1} $, we have
\begin{align*}
	\mathbb{P}\left( X_m \le \alpha \binom{ m - \ell }{k-\ell}\right) & \le \mathbb{P}(X_m \le X_0 - \lambda)
	\le  \exp \left( \frac{-\varepsilon^2 \binom{m}{k-\ell}^2}{8 \sum_{i=1}^{ m - \ell } \left( \binom{i-1}{k-\ell-1}  \right)^2}\right)\\
	& \le  \exp \left( \frac{-\varepsilon^2 \binom{m}{k-\ell}^2}{8 m \binom{m-1}{k-\ell-1}^2}\right) 
	= \exp \left( - \frac{ \varepsilon^2 m}{8(k-\ell)^2} \right).
\end{align*}
Thus, the claim follows.
\end{proof}
Therefore, the number of bad $m$-subsets is at most 
\begin{align*}
	 \sum_{T \in \binom{[n]}{ \ell }} \phi_T
	& \le \binom{n}{ \ell } \binom{ n - \ell }{ m - \ell } \exp \left( - \frac{\varepsilon^2 m }{8(k-\ell)^2} \right)\\
	 & = \binom{n}{m} \binom{m}{ \ell } \exp \left( - \frac{\varepsilon^2 m }{8(k-\ell)^2} \right) 
	\le \frac{1}{2} \binom{n}{m}.
\end{align*}
Hence, the proof of the lemma is completed.
\end{proof}

\begin{cor} \label{cor:1}
Let $k > \ell \ge 1$.
Given $\varepsilon >0$, let $M(k,\ell,\varepsilon)$ be the integer as defined in Lemma~\ref{lma:subgraph}.
Then 
\begin{align}
	\frac{ \ex_{\ell}(n,\mathcal{F})}{\binom{n}{k-\ell}} - \frac{\ex_{\ell}(m,\mathcal{F})}{\binom{m}{k-\ell}} \le \varepsilon  \nonumber
\end{align}
for all $n \ge m \ge M(k,\ell,\varepsilon)$ and all families $\mathcal{F}$ of $k$-graphs.
\end{cor}

\begin{proof}
Let $H_n$ be an $\mathcal{F}$-free graph with $\delta_{\ell}(H_n) = \ex_{\ell}(n, \mathcal{F})$.
Without loss of generality, we may assume that $\ex_{\ell}(n, \mathcal{F}) > \varepsilon {\binom{n}{k-\ell}}$ or else there is nothing to prove.
By Lemma~\ref{lma:subgraph}, there exists an $m$-subset~$S$ such that $\delta_{\ell}(H[S]) \ge \left( \frac{\ex_{\ell}(n, \mathcal{F})}{\binom{n}{k-\ell}} - \varepsilon \right) \binom{m}{k-\ell}$.
Note that $H[S]$ is $\mathcal{F}$-free and so $\delta_{\ell}(H[S]) \le \ex_{\ell}(m, \mathcal{F})$.
\end{proof}

\begin{proof}[Proof of Proposition~\ref{prp:existence}]
Define $a_n$ to be $\ex_{\ell} (n, \mathcal{F}) / \binom{n}{k-\ell}$.
Given $\varepsilon >0$, Corollary~\ref{cor:1} implies that there exists an integer $M$ such that $a_n - a_m < \varepsilon$ for all $n \ge m \ge M$.
Thus the sequence of $a_n$ convergences. 
\end{proof}

From Corollary~\ref{cor:1}, we also deduce the following statement, which will be useful.

\begin{cor} \label{cor:Turansubgraph}
Let $k > \ell \ge 1 $ be integers.
Let $\mathcal{F}$ be a family of $k$-graphs.
For every $\varepsilon >0$, there exists a finite family $\mathcal{F}' \subseteq \mathcal{F}$ with 
$\pi_{\ell}(\mathcal{F}) \le \pi_{\ell}(\mathcal{F}') \le \pi_{\ell}(\mathcal{F})+ \varepsilon$.
\end{cor}

\begin{proof}
Clearly, $\pi_{\ell}(\mathcal{F}) \le \pi_{\ell}(\mathcal{F}')$ for $\mathcal{F}' \subseteq \mathcal{F}$. 
Let $ \pi = \pi_{\ell}( \mathcal{F})$ and choose $m = m(\varepsilon)$ such that 
\begin{align*}
	\frac{\ex_{\ell}(m,\mathcal{F})}{ \binom{m}{k-\ell}} < \pi + \frac{\varepsilon}2 \textrm{ and } m \ge M(k, \ell , \varepsilon / 2),
\end{align*}
where $M(k,\ell, \varepsilon)$ is the integer as defined in Lemma~\ref{lma:subgraph}.
Define $\mathcal{F}'$ to be the family of $k$-graphs in $\mathcal{F}$ of order at most $m$.
Then $\ex_{\ell}(m,\mathcal{F}) = \ex_{\ell}(m,\mathcal{F}')$.
By Corollary~\ref{cor:1},
\begin{align*}
\frac{\ex_{\ell} (n,\mathcal{F}')}{\binom{n}{k-\ell}} 
< \frac{\ex_{\ell}(m,\mathcal{F}')}{\binom{m}{k-\ell}} + \frac{\varepsilon}2
=  \frac{\ex_{\ell}(m,\mathcal{F})}{\binom{m}{k-\ell}} + \frac{\varepsilon}2
\le \pi + \frac{\varepsilon}2+ \frac{\varepsilon}2 = \pi + \varepsilon.
\end{align*}
Therefore, $\pi_{\ell}(\mathcal{F}') = \lim_{n \rightarrow \infty} \frac{\ex_{\ell}(n,\mathcal{F}')}{\binom{n}{k-\ell}}  \le \pi + \varepsilon$.
\end{proof}

\begin{proof}[Proof of Proposition~\ref{prp:supersaturation}]
Given $\varepsilon >0$, $M$ and $\ell$ are assumed to be sufficiently large integers throughout this proof.
By Corollaries~\ref{cor:1} and~\ref{cor:Turansubgraph}, there exists a finite family $\mathcal{F}' \subseteq \mathcal{F}$ such that 
\begin{align*}
\ex_{\ell} (M,\mathcal{F}') / \binom{M}{k-\ell} <  \pi_{\ell}(\mathcal{F}') + \varepsilon /4 \le \pi_{\ell}(\mathcal{F}) + \varepsilon /2.
\end{align*}
Let $H_n$ be a graph with $\delta_{\ell}(H_n) > (\pi_{\ell}(\mathcal{F}) + \varepsilon ) \binom{n}{k-\ell}$ and $n \ge M$.
By Lemma~\ref{lma:subgraph}, the number of $M$-subsets~$S$ satisfying $\delta_{\ell}(H[S]) \ge (\pi_{\ell}(\mathcal{F}) + \varepsilon/2 ) \binom{M}{k-\ell} > \ex_{\ell}(M,\mathcal{F}')$ is at least $\binom{n}{M}/2$.
Thus, there exists a copy of a member of $\mathcal{F}'$ in $H[S]$ and so for some $F \in \mathcal{F}'$ on $f$ vertices, there are at least 
\begin{align}
	\frac{\binom{n}{M}/2}{|\mathcal{F}'| \binom{n-f}{M-f}} = \frac{\binom{n}{f}}{2 |\mathcal{F}'| \binom{M}f} \ge \delta \binom{n}f\nonumber
\end{align}
copies of $F$, where we set $\delta^{-1} = 2|\mathcal{F}'|  \max_{F' \in \mathcal{F}'}\{ \binom{M}{|V(F')|} \}$.

Define an $f$-graph $H'$ on the same vertex set of $H$ such that an edge of~$H'$ corresponds to a copy of~$F$ in~$H$.
Hence, $H'$ contains at least $\delta \binom{n}f$ edges.
By~\cite{MR0183654}, we know that there exists a copy of $\ell$-blow-up $K^f_f(L)$ of an edge $K_f^f$ in $H'$.
Each edge $e= (v_1, \dots, v_f)$ in $K^f_f(L)$ corresponds to an embedding of $F$ in $H$.
Also, each embedding of $F$ induced an colouring of $e$ in $K^f_f{s}$, namely a permutation of $V(F)$.
A result in Ramsey theory says that if $\ell$ is large enough, there exists a monochromatic $K^f_f(s)$ in $K^f_f(L)$.
Hence, $K^f_f(s)$ corresponds to a copy of $F(s)$ in $H$ and so $\pi_{\ell}(\mathcal{F}) = \pi_{\ell}(\mathcal{F}(s))$.
\end{proof}


\section{A lower bound on $\pi_{\ell}(\mathcal{F})$} \label{sec:lowerbound}

Given $k$-graphs $F$ and $H$, we say that $\phi: V(F) \rightarrow V(H)$ is \emph{homomorphism} if it preserves edges, i.e. $\phi(e) \in E(H)$ for all $e \in E(F)$.
Given a family $\mathcal{F}$ of $k$-graphs, we say that $H$ is \emph{$\mathcal{F}$-hom-free} if there is no homomorphism from $F$ to $H$ for all $F \in \mathcal{F}$.
Given $0 \le \ell < k$, define $\ex_{{\rm hom }, \ell}(n,\mathcal{F})$ to be the largest $\delta_\ell(H_n)$ in $\mathcal{F}$-hom-free $k$-graphs $H_n$ on $n$ vertices.
We further define $\pi_{{\rm hom }, \ell}(\mathcal{F}) = \limsup_{n \rightarrow \infty} \frac{\ex_{{\rm hom}, \ell }(n,\mathcal{F})}{\binom{n}{k- \ell}}$.
Note that $H$ is $\mathcal{F}$-hom-free if and only if $H(s)$ is $\mathcal{F}$-free for all $s \in \mathbb{N}$, where we recall $H(s)$ is the $s$-blow-up of $H$.
By `blowing-up', $\pi_{{\rm hom }, \ell}(\mathcal{F}) = \pi_{\ell}(\mathcal{F})$.

\begin{prp} \label{prp:homfree}
Let $0 \le \ell <k$.
Given a finite family $\mathcal{F}$ of $k$-graphs, we have $\pi_{{\rm hom },\ell}(\mathcal{F}) = \pi_{\ell}(\mathcal{F})$.
\end{prp}

\begin{proof}
Clearly,  $\pi_{{\rm hom }, \ell}(\mathcal{F}) \le \pi_{\ell}(\mathcal{F})$.
Let $\mathcal{G}$ be the family of all $k$-graphs $G$ such that there is a homomorphism from $F$ to $G$ for some $F \in \mathcal{F}$.
So  $\pi_{{\rm hom },\ell}(\mathcal{F}) = \pi_{\ell}(\mathcal{G})$.
Let $\varepsilon >0$.
By Corollary~\ref{cor:Turansubgraph}, there exists a finite family $\mathcal{G}' \subseteq \mathcal{G}$ with $ \pi_{\ell}(\mathcal{G}') \le \pi_{\ell}(\mathcal{G}) + \varepsilon/2$. 
By Proposition~\ref{prp:supersaturation}, $\pi_{\ell}(\mathcal{G}' ) = \pi_{\ell}(\mathcal{G}' (s)) $ for all $s \in \mathbb{N}$.
Let $s$ and $N$ be sufficiently large integers.
For $n \ge N$, let $H_n$ be a $k$-graph with $\delta_{\ell}(H_n) > (\pi_{{\rm hom },\ell}(\mathcal{F}) + \varepsilon) \binom{n}{k-\ell} >  \ex_{\ell}(n , \mathcal{G}'(s))$.
Hence $H_n$ contains a copy of $G(s)$ for some $G \in \mathcal{G}'$, which then implies that $H_n$ contians a member of $\mathcal{F}$.
Therefore $\pi_{{\rm hom },\ell}(\mathcal{F}) \ge \pi_{\ell}(\mathcal{F})$.
\end{proof}

We are going to bound $\pi_{\ell}(\mathcal{F})$ from below by $\pi(\mathcal{ L }_{ \ell - 1 }(\mathcal{F}) )$ proving Proposition~\ref{prp:lowerbound}.

\begin{proof}[Proof of Proposition~\ref{prp:lowerbound}]
Let $2 \le \ell < k $ be integers and let $\mathcal{F}$ be a family of $k$-graphs.
Set $\mathcal{ L } = \mathcal{ L }_{ \ell - 1 } ( \mathcal{F})$.
We may assume that $\pi = \pi(\mathcal{ L }) >0$, or else we are done.
Pick a constant $0 < \varepsilon < \pi/3$ and a sufficiently large integer $m \ge \varepsilon^{-1}$ such that there exists an $\mathcal{ L }$-hom-free $(k-\ell+1)$-graph $G_m$ on $[m]$ with $e(G_m) \ge (\pi-\varepsilon) \binom{m}{k-\ell+1}$ by Proposition~\ref{prp:homfree}.

Now we define a $k$-graph $H$ on $[n]$ as follows.
For every $\ell$-subset $S$ of~$[n]$, let $X_S$ be the independent uniform random variable on $[m]$.
The $k$-set $\{i_1, \dots, i_k\}$ is an edge in $H$ with $i_1 < \dots < i_k$ if and only if 
\begin{align*}
\left\{ X_{\{i_1, \dots, i_{ \ell - 1 },i_{ \ell -1+j}\}}: j \in [k-\ell+1]  \right\} & \textrm{ is an edge in $G_m$}. & (\dagger)
\end{align*}

Let $S$ be an $\ell$-subset of~$[n]$.
Suppose that $R$ is a $(k-\ell)$-subset of~$[n] \setminus S$ and $R \cup S = \{i_1, \dots, i_{k}\}$ with $i_1 < \dots < i_k$.
Recall that for an $\ell$-set $S'$, $X_{S'}$ is the independent uniform random variable on $[m]$.
By ($\dagger$), the probability of $R \cup S$ is an edge in $H$ is equal to the probability that a random $(k-\ell+1)$-subset~$S''$ of~$[m]$ forms an edge in $G_m$, where the elements of $S''$ are chosen independently uniformly at random with replacements.
Therefore
\begin{align*}
	\mathbb{P}(R  \in N^H(S) ) 
& \ge \frac{e(G_m)}{\binom{m}{k-\ell+1}} - O(1/m) 
 \ge  \pi-\varepsilon - O(1/m) \ge \pi- 3 \varepsilon /2
\end{align*}
as $m$ is large.
Thus, $\mathbb{E}(\deg^H(S)) \ge \left( \pi- 3 \varepsilon /2\right) \binom{ n - \ell }{k-\ell} \ge \left( \pi- 2 \varepsilon \right) \binom{ n }{k-\ell}$.

We now show that with high probability $\deg^H(S) \ge ( \pi - 3 \varepsilon )  \binom{ n  }{k-\ell}$ by considering the following martingale.
For each $1 \le i \le n$, define $Z_i$ to be the collection of $X_{S'}$ such that $S'$ is an $\ell$-subset of~$[n]$ and the largest element of $S'$ is~$i$.
So $Z_1,\dots,Z_i$ can be viewed as exposing the first $i$ vertices.
Define $Y_i = \mathbb{E}( \deg^H(S) |Z_1, \dots, Z_i)$.
Note that $\{ Y_i: i=0, 1, \dots, n \}$ is a martingale with 
\begin{align*}
Y_0  = \mathbb{E}(\deg^H(S))\ge \left( \pi - 2\varepsilon \right) \binom{ n }{k-\ell}
\end{align*}
and $Y_n = \deg^H(S)$.
Moreover, $|Y_i - Y_{i-1}| \le  \binom{i-1}{k-\ell-1}$.
Thus, by Theorem~\ref{thm:azuma} taking $\lambda = \varepsilon \binom{ n }{k-\ell}$ and $c_i = \binom{i-1}{k-\ell-1} $, we have
\begin{align*}
	& \mathbb{P}\left( \deg^H(S) \le ( \pi - 3 \varepsilon )  \binom{ n  }{k-\ell} \right) 
	 \le \mathbb{P}(Y_n \le Y_0 - \lambda)\\
	& \le  \exp \left( \frac{ -\varepsilon^2 \binom{ n  }{k-\ell}^2}{ 2 \sum_{i=1}^{ n } \left( \binom{i-1}{k-\ell-1}  \right)^2}\right)
	 \le  \exp \left( \frac{ -\varepsilon^2 \binom{n}{k-\ell}^2}{2 n \binom{ n-1}{k-\ell-1}^2}\right) 
	= \exp \left( - \frac{ \varepsilon^2 n}{2(k-\ell)^2} \right).
\end{align*}
Therefore, $\delta_{\ell}(H) \ge \left( \pi- 3\varepsilon \right) \binom{n}{k-\ell}$ with high probability for large $n$.

Next, we show that $H$ is $\mathcal{F}$-free.
Let $F$ be a member of $\mathcal{F}$ say with $|V(F)| = f$.
Let $T = \{i_1, \dots, i_f\}$ be an $f$-subset of~$[n]$ with $i_1 < \dots < i_f$.
By considering the map $\phi(i_j) = X_{\{i_1, \dots, i_{ \ell - 1 },i_{\ell +j}\}}$ for $j \in [k-\ell+1]$, the linked graph $N^H(\{i_1, \dots, i_{ \ell - 1 }\})$ is homomorphic to $G_m$.
Since $G_m$ is $\mathcal{ L }$-hom-free, $N^H(\{i_1, \dots, i_{ \ell - 1 }\})$ is $\mathcal{ L }_{\ell-1}(F)$-free and so $H[T]$ is $F$-free.
Therefore, $H$ is $\mathcal{F}$-free and so $\ex_{\ell}( n, \mathcal{F} ) \ge (\pi - 3 \varepsilon) \binom{n}{k-\ell}$.
Hence, $\pi_{\ell}(\mathcal{F}) \ge \pi$ as required.
\end{proof}

\section{Jumps}

From this section onward, unless stated otherwise we simply say a jump to mean a $\pi_{\ell}^k$-jump, where $k$ and $\ell$ are integers with $k > \ell >1$.
The following proposition shows the equivalent statement for $\alpha$ being a jump.

\begin{prp} \label{prp:jumpequ}
Let $k > \ell >1$ be integers. 
Let $0 \le \alpha <1$ and $ 0 < \delta \le 1-\alpha$.
The following statements are equivalent.
\begin{itemize}
	\item[\rm (S1)] Every family of $k$-graphs $\mathcal{F}$ satisfies $\pi_{\ell}(\mathcal{F}) \notin (\alpha, \alpha + \delta)$.
	\item[\rm (S2)] Every finite family of $k$-graphs $\mathcal{F}$ satisfies $\pi_{\ell}(\mathcal{F}) \notin (\alpha, \alpha + \delta)$.
	\item[\rm (S3)] For every $\varepsilon >0$ and every $M \ge k-1$, there exists an integer $N$ such that, for every $k$-graph $H_n$ with $n \ge N$ and $\delta_{ \ell }(H_n) \ge (\alpha + \varepsilon) \binom{n}{k-\ell}$, we can find a subhypergraph $H_m' \subseteq H_n$ with $\delta_{ \ell }(H_m') \ge (\alpha +\delta - \varepsilon) \binom{m}{k-\ell}$ for some $m \ge M$.
\end{itemize}
\end{prp}

\begin{proof}
Trivially, (S1)$\Rightarrow$(S2). 

(S2)$ \Rightarrow $(S1). 
Assume that there exist $0 \le \alpha<1$ and $\delta>0$ such that no finite family of $k$-graphs $\mathcal{F}$ satisfies $\pi_{\ell}(\mathcal{F}) \in (\alpha, \alpha + \delta)$.
Suppose that (S1) does not hold, so there exists a family of $k$-graph $\mathcal{F}$ such that $\pi_{\ell}(\mathcal{F}) \in (\alpha , \alpha + \delta)$.
Let $ \varepsilon =  ( \alpha + \delta - \pi_{\ell}(\mathcal{F}) )/2 $.
By Corollary~\ref{cor:Turansubgraph}, there exists a finite family $\mathcal{F}' \subseteq \mathcal{F}$ with $\alpha < \pi_{\ell}(\mathcal{F}) \le \pi_{\ell}(\mathcal{F}') \le \pi_{\ell}(\mathcal{F}) + \varepsilon < \alpha + \delta$, a contradiction.

(S3)$ \Rightarrow $(S1). 
Suppose that (S3) holds for some $0\le \alpha < 1$ and $\delta>0$.
Assume (S1) does not hold for a family $\mathcal{F}$ of $k$-graphs such that $\pi_{\ell}(\mathcal{F}) = \alpha + b$ for some $0 < b < \delta$.
Set $\varepsilon_0 = \min\{ b/2, (\delta-b)/2\}$.
Then, there exists an integer $m = m(\varepsilon_0)$ such that 
\begin{align}
{\ex_{\ell}(m,\mathcal{F})} \le (\pi_{\ell}(\mathcal{F}) + \varepsilon_0){\binom{m}{k-\ell}}. \label{eqn:S3S1}
\end{align}
Also, there exists an integer $n > N(\varepsilon_0, m)$ such that there exists an $\mathcal{F}$-free $k$-graph $H_n$ with $\delta_{\ell}(H_n) \ge ( \pi_{\ell}(\mathcal{F}) - \varepsilon_0) \binom{n}{k-\ell} \ge (\alpha + \varepsilon_0) \binom{n}{k-\ell}$ and $|V(H_n)| = n$.
By~(S3), there exists a subgraph $H_m' \subseteq H_n$ such that $ \delta_{\ell}(H_m') \ge (\alpha + \delta - \varepsilon_0) \binom{m}{k-\ell} > ( \pi_{\ell}(\mathcal{F}) +  \varepsilon_0) \binom{m}{k-\ell} $.
However, this contradicts \eqref{eqn:S3S1} as $H_m'$ is $\mathcal{F}$-free.

(S1)$ \Rightarrow $(S3).
Suppose that (S1) holds but (S3) does not. 
There exist $0 < \varepsilon < \delta$, $M \ge r-1$ and a sequence of $k$-graphs $H_{n_i}$ such that 
\begin{itemize}
	\item[(a$'$)] $|V(H_{n_i})| = n_i$ and $\delta_{\ell}(H_{n_i}) \ge (\alpha+\varepsilon) \binom{n_i}{k-\ell}$,
	\item[(b$'$)] $\delta_{\ell}(H') < (\alpha + \delta - \varepsilon) \binom{|H'|}{k-\ell}$ for all $H' \subseteq H_{n_i}$ with $|V(H')| \ge M$.
\end{itemize}
Let $\mathcal{F} = \{ G: |V(G)| \ge M $ and $G \not\subseteq  H_{n_i}$ for any $i\}$.
Note that $\mathcal{F}$ is not empty as $K_{M} \notin H_{n_i}$ for all~$i$. 
Note that $\pi_{\ell}(\mathcal{F}) \ge \alpha + \varepsilon$ by~(a$'$) and so $\pi_{\ell}(\mathcal{F}) \ge \alpha + \delta$ by~(S1).
Therefore, for every $m \ge M$, there exists an $\mathcal{F}$-free graph $G_m$ with $\delta_{\ell}(G_m) \ge (\alpha+\delta - \varepsilon) \binom{m}{k-\ell}$.
Since $G_m$ is $\mathcal{F}$-free, $G_m \subseteq H_{n_i}$ for some $i$ otherwise $G_m \in \mathcal{F}$.
This contradicts (b$'$).
\end{proof}

\section{Theorem~\ref{thm:nojump} for $\alpha =0$} \label{sec:0}

Here, we prove $0$ is not a jump.
Using (S3) in Proposition~\ref{prp:jumpequ}, the following statement implies that 0 is not a jump.
For every $\delta>0$, there exist $\varepsilon_0>0$ and $M_0>0$ , such that for every $N \ge k$, there exist $n \ge N$ and a $k$-graph $H_n$ satisfying
\begin{itemize}
	\item[(i)] $\delta_{\ell}(H_n) \ge \varepsilon_0 \binom{n}{k-\ell}$,
	\item[(ii)] $\delta_{\ell}(H'_m) < (\delta - \varepsilon_0) \binom{m}{k-\ell}$ for every subgraph $H'_m \subseteq H_n$ with $m \ge M_0$.
\end{itemize}
Hence, in order to prove Theorem~\ref{thm:nojump} for $\alpha = 0$, it is enough to show that there exist $k$-graphs satisfying (i) and~(ii).

When $\ell = k-1$, our construction is identical to that given by Mubayi and Zhao~\cite{MR2337241}.
To illustrate the key idea behind the construction, we consider the following simple case when $k=3$ and $\ell =2$.
Let $H$ be a $3$-graph on $n = pt$ vertices with vertex set $V = V_0 \cup \dots \cup V_{t-1}$ with $V_i \cap V_j = \emptyset$ for $i \ne j$ and $|V_i| = p$ for all $i$.
A $3$-set $S$ is an edge in $H$ if and only if either $|S \cap V_i| \le 1$ for all $i$ or $|S \cap V_i| = 2$ and $|S \cap V_{i+1}| = 1$ for some~$i$.
It is easy to check that $\delta_2 (H)  = n/t$.
Let $S$ be a subset of $V(H)$ with $|S| \ge t+1$.
If $| S \cap V_i | \ge 2$ for all~$i$, then $\delta_2 (H[S]) = \min \{|S \cap V_i| : i \in [n]\} \le |S|/t$.
Otherwise, $\delta_2 (H[S])\le 1$.
So (i) and (ii) hold by setting $\delta = 3/t$ and $\varepsilon_0 = 1/t$.

For general $k > \ell > 1$, we consider the $k$-graph $B(p,t,k,\ell)$ defined below.
\begin{dfn} \label{dfn:B}
For integers $p \ge 1$ and $t \ge k > \ell > 1$, define $B = B(p,t,k,\ell)$ to be the $k$-graph $(V,E)$ with the following properties:
\begin{enumerate}[\rm (a)]
	\item $V = V_0 \cup  \dots \cup  V_{t-1}$ with $V_i \cap V_j = \emptyset$ for $i \ne j$ and $|V_i| = p$ for all $i$.
	\item $E = E_1 \cup E_2$.
	\item $E_1 = \{ S \in \binom{V}{k} : |S \cap V_i| < \ell $ for all $i \}$.
	\item $E_2 = \{ S \in \binom{V}{k} : $ there exists an integer $i$ such that $|S \cap V_i| = \ell $ and $|S \cap V_{i+j}| = 1$ for each $j \le k-\ell \}$, (here $V_t = V_0$).
\end{enumerate}
\end{dfn}

Clearly, $|B| = tp$.
We now show that $\delta_{\ell}(B) = p^{k-\ell}$.

\begin{fact} \label{fact:B}
For integers $p \ge 1$ and $t \ge k > \ell > 1$, $\delta_{\ell}(B(p,t,k,\ell)) = p^{k-\ell}$.
\end{fact}

\begin{proof}
Let $T$ be an $\ell$-subset of~$V$.
If $T \subseteq V_{i_0}$ for some $i_0$, then $\deg(T) = p^{k-\ell}$ by~(d).
If $|T \cap V_{i}|< \ell $ for all $i$, then $\deg(T) \ge p^{k-\ell}$ by (c) and picking one element in each $V_i$ with $V_i \cap T = \emptyset$ as $t \ge k$.
\end{proof}

\begin{proof}[Proof of Theorem~\ref{thm:nojump} for $\alpha =0$]
Let $k > \ell >1$ and $0 < \delta < \left( 4 k^{-1}(k-\ell)^2 \right)^{k-\ell}$.
Define $\varepsilon_0$, $t$ and $M_0$ to be constants such that 
\begin{align}
\varepsilon_0^{k-\ell}  & =  \frac{\delta}{1+(4(k-\ell)^2)^{k-\ell}} \le  k^{-(k-\ell)}  < 1/2, \label{eqn:e_0}\\
 t & = \left\lfloor \frac{1}{\varepsilon_0}\right\rfloor \ge k,  \label{eqn:t}\\
 M_0 &=  \ell t \ge \frac{ \ell }{2 (k-\ell)\varepsilon_0}. \label{eqn:M_0}
\end{align}
For $p \ge k-1$, set $n = pt$ and let $H_n = B(p,t,k,\ell)$ be the $k$-graph as defined in Definition~\ref{dfn:B}.
By Fact~\ref{fact:B}, $\delta_{\ell}(H_n) = p^{k-\ell}$, so
\begin{align}
	\frac{ \delta_{\ell} ( H_n ) }{\binom{n}{k-\ell}} = \frac{p^{k-\ell}}{\binom{pt}{k-\ell}} \ge t^{-(k-\ell)} \ge \varepsilon_0^{k-\ell}. \nonumber
\end{align}
To complete the proof, our task is to show that (ii) holds, that is, for every $m$-subset~$S$ of~$ V$ with $m >M_0$, $\delta_{\ell}( H_n [S] ) \le (\delta - \varepsilon_0^{k-\ell}) \binom{m}{k-\ell}$.

Let $S$ be an $m$-subset of~$V$ with $m >M_0$ and let $H' = H_n[S]$.
Suppose that $\delta_{\ell}(H') > (\delta - \varepsilon_0^{k-\ell}) \binom{m}{k-\ell}$.
Set $S_i = S \cap V_i$ for all $i$.
Since $m > M_0 = \ell t$, there exists an integer $i_0$ and an $\ell$-subset~$R_0$ of~$S_{i_0}$.
By~\eqref{eqn:e_0}, we have
\begin{align*}
 \deg^{H'}(R) & \ge \delta_{\ell}(H') > (\delta - \varepsilon_0^{k-\ell}) \binom{m}{k-\ell} \\ 
&  = \left(4(k-\ell)^2 \varepsilon_0 \right)^{k-\ell}  \binom{m}{k-\ell} \ge \left( 2 (k-\ell)\varepsilon_0 m \right)^{k-\ell},
\end{align*}
where we take $\binom{m}{k-\ell} \ge \left(\frac{m}{2(k-\ell)} \right)^{k-\ell}$.
Recall Definition~\ref{dfn:B}(d) that $N^{H_n}(R) \subseteq V_{i_0+1} \times \dots \times V_{i_0+k-\ell}$.
Hence, there exists an integer $j \le k-\ell$ such that $|S_{i_0+j}| \ge 2 (k-\ell)\varepsilon_0 m $.
Since $2 (k-\ell)\varepsilon_0 m \ge 2(k-\ell) \varepsilon_0 M_0 \ge \ell $ by~\eqref{eqn:M_0}, there is an $\ell$-subset~$R_{i_0+j}$ of~$S_{i_0+j}$.
By repeating this argument to $i_0+j$, we may conclude that there are at least $t/(k-\ell)$ integers $0 \le i \le  t-1$ such that $|S_i| \ge 2 (k-\ell)\varepsilon_0 m$.
Therefore, by \eqref{eqn:t}
\begin{align*}
	m \ge \frac{t}{k-\ell} \times 2 (k-\ell)\varepsilon_0 m
	\ge 2 \left( \frac1{\varepsilon_0} - 1\right) \varepsilon_0 m > m
\end{align*}
as $\varepsilon_0 < 1/2$ by~\eqref{eqn:e_0}.
This is a contradiction, so $\delta_{\ell}(H_n[S]) < (\delta - \varepsilon_0^{k-\ell}) \binom{m}{k-\ell}$ for all $m$-subset~$S$ of~$V$ and all $m > M_0$.
Therefore, $0$ is not a $\pi_{\ell}^k$-jump for all $k > \ell >1$.
\end{proof}

\section{Proof of Theorem~\ref{thm:nojump} for $0< \alpha <1$} \label{sec:nojump}

As in Section~\ref{sec:0}, we are going to construct $k$-graphs $H_n$ that contradict (S3) in Proposition~\ref{prp:jumpequ}.
The structure of $H_n$ is based on $B = B(p,t,k,\ell)$.
(The formal definition of $H_n$ will be given later.)
So $\delta_{\ell}(H_n) = \deg^{H_n}(T)$ for some $\ell$-subset~$T$ of~$V_0$.
In order to ensure that $\deg^{H_n}(T) \ge \alpha \binom{n}{k - \ell}$,  we modify Definition~\ref{dfn:B}(d) so that $N^{H_n}(T)$ contains all $(k-\ell)$-subsets~$S$ of~$V_1 \cup \dots \cup V_a$ with $|S \cap V_i |< \ell$ for all $i \in [a]$.

The following function will turn out to be very useful when calculating $\delta_{\ell}(H_n)$.
For integers $a$ and $k > \ell >1$, let $\mathcal{A} = \mathcal{A}(a,k,\ell)$ be the set of all ordered $a$-tuples $(q_1, \dots, q_a)$ in $\{0,1 \dots,\ell-1\}^a$ such that $\sum q_i = k-\ell$.
Given integers $n_1, \dots, n_a$, define the function 
\begin{align*}
f(n_1, \dots, n_a) = \sum_{(q_1, \dots, q_a) \in \mathcal{A}} \prod_{i=1}^a \binom{n_i}{q_i}.
\end{align*}
Thus, $f(n_1, \dots, n_a)$ is the number of ways of choosing $k-\ell$ elements from vertex sets $U_1$, \dots, $U_a$ with $|U_i| = n_i$ such that at most $ \ell -1 $ elements are chosen from each~$U_i$.
Write $f(n_0;a)$ for $f(n_1, \dots, n_a)$ if $n_i = n_0$ for all $i \in [a]$.
Note that 
\begin{align*}
f(n_1, \dots, n_a) \le \binom{\sum n_i}{k-\ell}.
\end{align*}
Moreover, from the definition of~$f$, for integers $a,b\ge 1$ with $a/b < 1$ we also have
\begin{align*}
f(n_0/b;a) 
& \ge  \binom{\frac{an_0}{b}}{k-\ell} - a \binom{\frac{n_0}{b}}{ \ell }\binom{\frac{a n_0}{b}- \ell }{k-2\ell} \\
& \ge \binom{\frac{an_0}{b}}{k-\ell} - \frac{a^{k-2 \ell+1}}{b^{k-\ell}} \binom{n_0}{ \ell }\binom{n_0- \ell}{k-2\ell} \\
& = \binom{\frac{an_0}{b}}{k-\ell} - \left(\frac{a}{b} \right)^{k-\ell} a^{-(\ell - 1)} \binom{k-\ell}{ \ell } \binom{n_0}{k-\ell}.
\end{align*}
Thus, given integers $k> \ell >1$, a rational $q \in (0,1)$ and a constant $\varepsilon >0$ there exist integers $0 < a < b < N_0$ such that $a /b = q$, $a+k < b$ and for all $n' \ge N_0$,
\begin{align}
	 q^{k-\ell} - \varepsilon 
= \left(\frac{a}{b}\right)^{k-\ell} - \varepsilon 
\le \frac{f(n'/b;a)}{\binom{n'}{k-\ell}} 
\le  \frac{ \binom{a n' / b}{k-\ell} }{ \binom{n'}{k-\ell} }
\le \left(\frac{a}{b}\right)^{k-\ell} 
= q^{k-\ell}. \nonumber
\end{align}

We now give the proof of Theorem~\ref{thm:nojump} when $0<\alpha <1$, that is, no $\alpha \in (0,1)$ is a jump.

\begin{proof}[Proof of Theorem~\ref{thm:nojump} for $0< \alpha <1$]
Note that the set of $q^{k-\ell}$ for rationals $q \in (0,1)$ is dense in $(0,1)$.
Hence, it is enough to show that $q^{k-\ell}$ is not a jump for all rationals $q \in (0,1)$.
Given $\delta>0$, let $\varepsilon_0$, $t$ and $M_0$ as in \eqref{eqn:e_0}, \eqref{eqn:t} and \eqref{eqn:M_0} respectively.
Pick $0<a<b<N_0$ such that
\begin{align}
	a/b = q, \qquad	
a+k < b, \qquad 
q^{k-\ell} - \varepsilon_0^{k-\ell} \le \frac{f(n'/b;a)}{\binom{n'}{k-\ell}}  \le q^{k-\ell}
	\label{eqn:f(n/b;a)lowerbound}
\end{align}
for all $n' > N_0$.
Set $\varepsilon =  (\varepsilon_0/ b )^{k-\ell}/4$.
Define $M$ to be an integer such that
\begin{align}
M  \ge  \max\left\{ 
2(k-\ell) \left( \frac{\binom{M_0}{k-\ell} }{\delta - \varepsilon_0^{k-\ell}} \right)^{\frac{1}{k-\ell}}
, \ 
\frac{ b \ell (k-\ell)}{q \ln \left(1+ \frac12 \left(\frac{\varepsilon_0}{q}\right)^{k-\ell} \right) }, \ 
N_0,
\
\frac{  b \ell }{1-q}
\right\} .
\label{eqn:eandM}
\end{align}
For integers $p >k$, set $n = btp$.
Let $H = (V,E)$ be the $k$-graph with the following properties:
\begin{enumerate}[(a$'$)]
	\item $V = V_0 \cup V_1 \cup  \dots \cup  V_{b-1}$ with $V_i \cap V_j = \emptyset$ for $i \ne j$ and $|V_i| = tp$ for all $i$.
	\item $E = E_1 \cup E_2 \cup \bigcup_{j=0}^{b-1}E'_j$.
	\item $E_1 = \{ S \in \binom{V}{k} : |S \cap V_i| < \ell $ for all $i \}$.
	\item $E_2$ is the set of $k$-subsets~$S$ of~$V$ such that 	
\begin{align*}
	|S \cap V_i|
	\begin{cases}
	= \ell & \textrm{if $i = i_0$,}\\
	< \ell & \textrm{if $i_0+1 \le i \le  i_0+a$,}\\
	= 0 & \textrm{otherwise,}
	\end{cases}
\end{align*}
	for some $i_0$, where we take $V_{b+j} = V_j$.
	\item $E'_i = E(H[V_i])$, which is isomorphic to $B(p,t,k,\ell)$, for all $i$.	
\end{enumerate}
Note that the definition of the edge set $E_2$ is different to the one defined in $B(p,t,k,\ell)$.
Pick $p$ large such that $n \ge M$.
First, we are going to show that 
\begin{align*}
\delta_{\ell}(H)  = f(n/b;a) + p^{k-\ell}.
\end{align*}
Let $T$ be an $\ell$-subset of~$V$.
If $T \subseteq V_{i_0}$ for some $i_0$, then Fact~\ref{fact:B} implies that $\deg^{H[V_{i_0}]}(T) \ge p^{k-\ell}$ with equality holds for some $\ell$-subset $T$ of $V_{i_0}$. 
By (d$'$), the number of $(k-\ell)$-subsets $S$ of~$N^H(T) \setminus \binom{V_{i_0}}{k-\ell}$ is $f(n/b;a)$.
Thus, $\deg^H (T) \ge f(n/b;a) +p^{k-\ell}$ with equality holds for some $\ell$-subset $T$ of $V_{i_0}$.
If $|T \cap V_{i}|< \ell $ for all $i$, then $\deg^H (T)  \ge f(n/b;b-\ell) \ge f(n/b;a)+p^{k-\ell}$ by (c$'$) and picking at most $ \ell -1 $ elements in each $V_i$ with $V_i \cap T = \emptyset$.
Hence, $\delta_{\ell}(H)  = f(n/b;a) + p^{k-\ell}$ as claimed.
Note by~\eqref{eqn:t} and the definition of $\varepsilon$ that 
$$p^{k-\ell} = \left( \frac{n}{bt} \right)^{k-\ell} \ge \left( \frac{\varepsilon_0}{b} \right)^{k-\ell} \binom{n}{k-\ell} > 2 \varepsilon \binom{n}{k-\ell}.$$
Recall that $n \ge M$, so we have by~\eqref{eqn:f(n/b;a)lowerbound}
\begin{align*}
\delta_{\ell}(H) & = f(n/b;a) +p^{k-\ell} \\
 & >  \left( q^{k-\ell} - \varepsilon \right)\binom{n}{k-\ell} + 
2\varepsilon \binom{n}{k-\ell} = 
\left( q^{k-\ell} + \varepsilon \right)\binom{n}{k-\ell}.
\end{align*}
By Proposition~\ref{prp:jumpequ}(S3), to complete the proof, our task is to show that for every $m$-subset~$S$ of~$V$ with $m >M$, $\delta_{\ell}(H[S]) \le (q+\delta - \varepsilon_0^{k-\ell}) \binom{m}{k-\ell}$.
Let $S$ be an $m$-subset of~$V$ with $m >M$ and let $H' = H[S]$.
Suppose that $\delta_{\ell}(H') > (q+\delta - \varepsilon_0^{k-\ell}) \binom{m}{k-\ell}$.
Set $S_i = S \cap V_i$ and $m_i = |S_i|$ for all $i$.

\begin{clm} \label{clm:i_0}
There exists $i_0$ such that $m_{i_0} \ge \ell $ and $\sum_{j=i_0+1}^{i_0+a} m_j < q m + b \ell $.
\end{clm}

\begin{proof}[Proof of claim]
If $\sum_{j=i+1}^{i+a} m_j > qm$ for all $i$, then by averaging we get $|S| > qm b/a  = m$, a contradiction.
Hence, we may assume without loss of generality that $\sum_{j=1}^{a} m_j \le qm$ and $m_0 < \ell $.
Let $i_0<b$ be the largest integer such that $m_{i_0} \ge \ell $.
Note that such $i_0$ exists as $|S| = m \ge M > b \ell$.
Then 
\begin{align*}
	\sum_{j = i_0+1}^{i_0+a} m_j \le \sum_{j = i_0+1}^{b-1} m_j + |S_0| + \sum_{j = 1}^{a} m_j \le (b-1)(\ell - 1) + qm.
\end{align*}
Hence, the claim holds.
\end{proof}

Without loss of generality, we may assume that Claim~\ref{clm:i_0} holds with $i_0 = 0$.
Let $R$ be an $\ell$-subset of~$S_0$.
If $|S_0| \ge M_0$, then 
\begin{align*}
\deg^{H[S_0]}(R) < ( \delta - \varepsilon_0^{k-\ell}) \binom{|S_0|}{k-\ell}  \le ( \delta - \varepsilon_0^{k-\ell}) \binom{m}{k-\ell}
\end{align*}
 as $H[V_0]$ is isomorphic to $B(p,t,k,\ell)$ by~(e$'$).
Therefore
\begin{align*}
\deg^{H[S_0]}(R)  &  \le \max \left\{ ( \delta - \varepsilon_0^{k-\ell}) \binom{m}{k-\ell}, \binom{M_0}{k-\ell} \right\}.
\end{align*}	
By~\eqref{eqn:eandM}, $( \delta - \varepsilon_0^{k-\ell}) \binom{m}{k-\ell} \ge  ( \delta - \varepsilon_0^{k-\ell}) \left(\frac{m}{2(k-\ell)}\right)^{k-\ell} \ge \binom{M_0}{k-\ell}$.
In summary, 
\begin{align}
\deg^{H[S_0]}(R) < ( \delta - \varepsilon_0^{k-\ell}) \binom{m}{k-\ell}. \label{eqn:degS0}
\end{align}
Recall (d$'$) that if $T \in N^H(R)$ with $T \not \subseteq S_0$, then $T \subseteq S_1 \cup \dots \cup S_a$ and $|T \cap S_i |< \ell $ for all $i \in [a]$.
Since $|S_j| = m_j$ and $\sum_{j=1}^{a} m_j < \frac{am}{b} + b \ell $, the number of $S \in N^{H'}(R) \setminus \binom{V_0}{k-\ell}$ is equal to 
\begin{align*}
 f(m_1, \dots, m_a) \le \binom{\sum_{j=1}^{a} m_j}{k-\ell}  < \binom{am/b+  b \ell }{k-\ell} \le \left( q + \frac{  b \ell }m \right)^{k-\ell} \binom{m}{k-\ell}
\end{align*}
as $q + b \ell /m \le q + b \ell /M \le 1$.
Therefore, together with~\eqref{eqn:degS0} and \eqref{eqn:eandM}, 
\begin{align*}
	\deg^{H'}(R) & \le \deg^{H[S_0]}(R) + \left| N^{H'}(R) \setminus \binom{S_0}{k-\ell} \right|
\\
 &\le  \left( \delta - \varepsilon_0^{k-\ell} + \left(q + \frac{  b \ell }m \right)^{k-\ell} \right) \binom{m}{k-\ell}\\
		& \le\left(  \delta - \varepsilon_0^{k-\ell} +   q^{k-\ell}  \exp\left( \frac{ b \ell (k-\ell)}{qM} \right) \right) \binom{m}{k-\ell} \\
		& \le\left( \delta - \varepsilon_0^{k-\ell} +   q^{k-\ell} \left( 1+ \frac12 \left( \frac{\varepsilon_0}{q} \right)^{k-\ell}  \right) \right) \binom{m}{k-\ell} \\
& = \left(q^{k-\ell} +\delta - \frac12 \varepsilon_0^{k-\ell}   \right) \binom{m}{k-\ell} \le \left(q^{k-\ell} +\delta - \varepsilon   \right) \binom{m}{k-\ell} .
\end{align*}
Therefore, $\delta_{\ell}(H[S]) <\left( q^{k-\ell} + \delta - \varepsilon \right) \binom{m}{k-\ell}$ for all $m$-subsets~$S$ of~$V$ and all $m > M$.
Hence, $q^{k-\ell}$ is not a $\pi_{\ell}^k$-jump.
\end{proof}

\section{Acknowledgement}
The authors would like to thank the anonymous referee for pointing out an error in the previous version of the manuscript.

\end{document}